\documentclass[a4paper,11pt,titlepage,twoside]{article}

\usepackage{graphicx}
\usepackage[T1]{fontenc} 
\usepackage[utf8]{inputenc}
\usepackage[english]{babel}
\usepackage{soul}
\usepackage{amsfonts}
\usepackage{amsmath}
\usepackage{amsthm}
\usepackage{amssymb}
\usepackage{mathrsfs}
\usepackage[top=5.5cm, bottom=3cm, left=4cm, right=4cm]{geometry}
\usepackage{setspace}
\usepackage{afterpage}
\usepackage{extarrows}
\usepackage{fancyhdr}
\usepackage{titlesec}
\usepackage{enumitem} \setlist{nosep}
\usepackage[pdftex,breaklinks,colorlinks,linkcolor=blue,
anchorcolor=blue]{hyperref}

\theoremstyle{definition}
\newtheorem{defin}{Definition}[section]
\theoremstyle{plain}
\newtheorem{theo}[defin]{Theorem}
\newtheorem{lem}[defin]{Lemma}
\newtheorem{pro}[defin]{Proposition}

\theoremstyle{definition}
\newtheorem{exm}[defin]{Example}
\newtheorem{rem}[defin]{Remark}
\numberwithin{equation}{section}

\newcommand{\D}{\mathfrak{D}}
\renewcommand{\H}{\mathfrak{H}}

\newcommand{\n}[1]{\|#1\|}
\newcommand{\nor}{\|\cdot\|}
\newcommand{\nosw}[1]{\|#1\|_{\s+\w}}

\renewcommand{\l}{\langle}
\renewcommand{\r}{\rangle}

\newcommand{\R}{\mathbb{R}}
\newcommand{\C}{\mathbb{C}}

\newcommand{\pint}{\l\cdot,\cdot\r}
\newcommand{\pin}[2]{\l#1 , #2\r}

\newcommand{\no}{\noindent}
\newcommand{\ol}{\overline}

\newcommand{\sub}{\subseteq}

\newcommand{\mez}{\frac{1}{2}}
\renewcommand{\t}{\mathfrak{t}}
\newcommand{\s}{\mathfrak{s}}
\newcommand{\w}{\mathfrak{w}}

\newcommand{\RR}{\Sigma}

\newcommand{\af}{\mathfrak{a}}
\newcommand{\bb}{\mathfrak{b}}

\renewcommand\labelenumi{\emph{(\roman{enumi})}}
\renewcommand\theenumi\labelenumi

\makeatletter
\newcommand{\oset}[2]{%
	{\mathop{#2}\limits^{\vbox to 2.7\ex@{\kern-\tw@\ex@
				\hbox{\footnotesize  #1}\vss}}}}
\makeatother

\titleformat{\section}
{\normalfont\fillast \fontsize{12}{15}\scshape}{\thesection.}{0.8em}{}

\titleformat{\subsection}
{\normalfont\fillast \fontsize{11}{12}\scshape}{\thesubsection.}{0.8em}{}

\begin{document}

\thispagestyle{plain}

\begin{center}
	\large
	{\uppercase{\bf A Lebesgue-type decomposition for \\ non-positive sesquilinear forms}} \\
	\vspace*{0.5cm}
	ROSARIO CORSO
\end{center}

\normalsize 
\vspace*{1cm}	

\small 

\begin{minipage}{11.8cm}
	{\scshape Abstract.} 
	A Lebesgue-type decomposition of a (non necessarily non-negative) sesquilinear form with respect to a non-negative one is studied. This decomposition consists of a sum of three parts: two are dominated by an absolutely continuous form and a singular non-negative one, respectively, and the latter is majorized by the product of an absolutely continuous and a singular non-negative forms.\\
	The Lebesgue decomposition of a complex measure is given as application.
\end{minipage}

\vspace*{.5cm}

\begin{minipage}{11.8cm}
\end{minipage}

\vspace*{.5cm}

\begin{minipage}{11.8cm}
\end{minipage}

\vspace*{1cm}
\normalsize

\section{Introduction}

In \cite{Simon} Simon proved a decomposition of a non-negative form defined on a dense subspace of a Hilbert space into the sum of two non-negative forms such that one is the greatest non-negative form which is smaller than the form and {\it closable}. The second form is referred as the {\it singular} part.

However, the definition of singular non-negative form, in terms of sequences, goes back to Koshmanenko \cite{Kos_sing} (see also his book \cite{Kos} dedicated to singular forms).

Simon, again in \cite{Simon}, stated the correspondent decomposition of a non-negative form $\t$ into the sum of a closable (or, with another terminology, {\it absolutely continuous}) form $\t_a$ and a singular form $\t_s$ with respect to a second non-negative form $\w$ (see also \cite{Kos}). In this setting, $\t$ and $\w$ are defined on a common complex vector space. 
The study of this last so-called {\it Lebesgue decomposition} was continued by Hassi, Sebestyén, De Snoo in \cite{HSdeS}. 
Their framework involves the notion of parallel sum of forms, which is inspired by the one for non-negative operators used by Ando \cite{Ando}. A proof with a different approach was developed by Sebestyén, Tarcsay and Titkos \cite{STT}.

The Lebesgue decomposition of non-negative forms, as the name suggests, is inspired to the classical Lebesgue decomposition of non-negative measures (or, in more generality, additive set functions). Moreover, these notions are related. Indeed, a non-negative measure induces a non-negative form and the absolutely continuous parts are in correspondence, as well as the singular parts (see \cite[Theorem 5.5]{HSdeS} and also \cite[Theorems 3.2 and 3.4]{STT}). 

Recently, Di Bella and Trapani \cite{Tp_DB} have given a notion of regularity and singularity for a (non-necessarily non-negative) sesquilinear form with respect to a non-negative one and then they proved a correspondent Lebesgue decomposition theorem.
More precisely, let $\w,\t$ be forms on $\D$, $\w$ being non-negative. We denote by $M(\t)$ the set of non-negative sesquilinear forms $\s$ satisfying the inequality $|\t(\xi,\eta)|\leq \s[\xi]^\mez \s[\eta]^\mez$ for all $\xi,\eta \in \D$. Then a sesquilinear form $\t$ is $\w$-{\it regular} if there exists $\s\in M(\t)$ such that $\s$ is $\w$-absolutely continuous. On the other hand, $\t$ is {\it $\w$-singular} if for every $\phi\in \D$ there exists a sequence $(\phi_n)\subset\D$ with
\begin{equation*}
\label{def:sing}
\lim_{n\to +\infty} \w[\phi_n]=0 \;\;\text{ and }\; \lim_{n\to +\infty} \t[\phi-\phi_n] =0.
\end{equation*}
Furthermore, Theorem 4.3 of \cite{Tp_DB} states that if $M(\t)\neq \varnothing$, then $\t=\t_r+\t_s$ where $\t_r$ is $\w$-regular and $\t_s$ is $\w$-singular.

In this paper Di Bella and Trapani's theorem is reconsidered. First of all, in analogy to the notion of $\w$-regularity, one can give a notion of singularity of a form $\t$ (coherent to the classical one in the non-negative case) as follows
\begin{equation}
\tag{ss}\label{def_intr_ss}
\text{$\exists\s\in M(\t)$ such that $\s$ is $\w$-singular.}
\end{equation}
This idea is supported by the following fact from the Theory of Measure.  
If $\mu,\nu$ are (complex) measure on the same $\sigma$-algebra and $\nu$ is non-negative, then $\mu$ is $\nu$-absolutely continuous (resp. $\nu$-singular) if and only if it is dominated by an $\nu$-absolutely continuous (resp. $\nu$-singular) non-negative measure. \\
%on $(\mathcal{A},\RR)$ and $|\mu(A)|\leq \tau(A)$ for all $A\in \RR$.
Nevertheless, condition (\ref{def_intr_ss}) does not always hold for the singular part of a form in \cite{Tp_DB} (see Remark \ref{cexm}), but actually it is a stronger notion. For this reason, we give to a form $\t$ satisfying (\ref{def_intr_ss}) the name of $\w$-{\it strongly singular form}. \\
However, it turns out (Theorem \ref{Leb_th}) that every sesquilinear form  $\t$ such that $M(\t)\neq \varnothing$ can be decomposed as $\t=\t_{r}+\t_m+\t_{ss}$, where $\t_{r}$ is the $\w$-regular part, $\t_{ss}$ the $\w$-strongly singular part and $\t_m$ is a form (called $\w$-{\it mixed}) which is dominated by the product of a non-negative $\w$-absolutely continuous form and a non-negative $\w$-singular form. This is the version of the Lebesgue decomposition that we states in the present article.

The organization of this paper is as follows. In Sections \ref{sec_pre} we establish some properties and characterizations of the forms considered above, as well as some examples. Under simple conditions on the values of a $\t$ (Proposition \ref{pro_N_t}) one can see cases where a $\w$-mixed form is identically zero (for example assuming the condition of non-negativity) or that the notions of $\w$-singularity and $\w$-strongly singularity are equivalent. Section \ref{sec_Leb} contains the Lebesgue decomposition of forms as stated above, and shows also that it is not the same if one chooses a different non-negative dominant form $\s\in M(\t)$. Finally, relations between measures and forms are investigated in Section \ref{sec_meas}.  In particular, the Lebesgue decomposition of a complex measure with respect to a non-negative one is proved through sesquilinear forms.

\section{Preliminaries}
\label{sec_pre}

To make the topic on sesquilinear forms as self-contained as possible we begin recalling basic notions and properties. 
A {\it sesquilinear form} $\t$ on a complex vector space $\D$ (called the {\it domain} of $\t$) is a map $\D\times \D\to \C$ which is linear in the first component and anti-linear in the second one. The map $\D\to \C$ defined by $\phi \mapsto \t[\phi]:=\t(\phi,\phi)$ is the {\it quadratic form} associated to $\t$. The {\it polarization identity}
$$
\t(\phi,\psi)=\frac{1}{4} \sum_{k=0}^{3} i^k\t[\phi+i^k\psi], \qquad \forall \phi,\psi \in \D
$$
connects quadratic and sesquilinear forms. The {\it scalar multiple} $\alpha \t$, with $\alpha \in \C$, is defined as
\begin{align*}
(\alpha \t)(\phi,\psi)&:=\alpha\t(\phi,\psi), \qquad  \phi,\psi \in \D.
\end{align*}
Given two sesquilinear forms $\t_1,\t_2$ on $\D_1$ and $\D_2$, respectively, the {\it sum} $\t_1+\t_2$ is the sesquilinear form  
\begin{align*}
(\t_1+\t_2)(\phi,\psi)&:=\t_1(\phi,\psi)+\t_2(\phi,\psi),\qquad  \phi,\psi \in \D_1\cap \D_2.
\end{align*}

\no Classic forms associated to a sesquilinear form $\t$ on $\D$ are:
\begin{itemize}
	\item  the {\it adjoint} $\t^*$ of $\t$, defined as
	$$
	\t^*(\phi,\psi)=\ol{\t(\psi,\phi)}, \qquad  \phi, \psi\in \D;
	$$
	\item the {\it real part} $\Re \t$ of $\t$, defined as $\Re \t := \mez(\t+\t^*)$;
	\item the {\it imaginary part} $\Im \t$ of $\t$, defined as $\Im \t := \frac{1}{2i}(\t-\t^*)$.
\end{itemize}
A sesquilinear form $\t$ on $\D$ is called {\it symmetric} if $\t=\t^*$ and, in particular, {\it non-negative} (in symbol $\t\geq 0$) if $\t[\phi]\geq 0$ for all $\phi \in \D$. In this latter case the {\it Cauchy-Schwarz} and {\it triangle inequalities} hold; i.e.,
\begin{align*}
|\t(\phi,\psi)|&\leq \t[\phi]^\mez\t[\psi]^\mez,\\
\t[\phi+\psi]^\mez &\leq \t[\phi]^\mez + \t[\psi]^\mez, \qquad \forall \phi,\psi\in \D.
\end{align*}

If $\s_1$ and $\s_2$ are non-negative sesquilinear forms on $\D$, we write $\s_1\leq\s_2$ when $\s_1[\phi]\leq \s_2[\phi]$ for all $\phi \in \D$.

If $\D$ is a subspace of a Hilbert space $\H$ with inner product $\pint$ and corresponding norm $\nor$, a sesquilinear form $\t$ on $\D$ satisfying for some $C\geq 0$, 
$|\t(\phi,\psi)|\leq C \n{\phi}\n{\psi}$ for all $\phi,\psi \in \D,$
is called ({\it normed}) {\it bounded} on $\H$. For this form there exists a bounded operator $T$ on $\H$ such that $\t(\phi,\psi)=\pin{T\phi}{\psi}$, for all $\phi,\psi \in \D$. Moreover, if $\D$ is dense in $\H$, then $T$ is unique (with norm not greater than $C$) and $\t$ can be extended to a bounded form defined on the whole of $\H$, called the {\it closure} of $\t$. \\

For reader's convenience we also summarize the definitions presented and motivated in the Introduction, part of which are taken from \cite{Tp_DB}.

Let $\D$ be a complex vector space and $\t,\w$ sesquilinear forms on $\D$. Throughout the paper $\w$ will be non-negative.\\ 
We write $M(\t)$ for the set of non-negative sesquilinear forms $\s$ on $\D$ satisfying
\begin{equation*}
\label{def_M_t}
|\t(\phi,\psi)|\leq \s[\phi]^\mez\s[\psi]^\mez, \qquad \forall \phi,\psi\in \D.
\end{equation*}

\no The set $M(\t)$ is not empty if and only if there exists a form $\s\geq 0$ on $\D$ such that $|\t[\phi]|\leq \s[\phi]$ for all $\phi \in \D$ (it follows by an argument like in the proof of Lemma 11.1 in \cite{Schm}).

\no The following definitions will also be needed in the sequel:
\begin{itemize}
	\item if $\t$ is non-negative, $\t$ is  {\it $\w$-absolutely continuous} (in symbols $\t\ll\w$) if for every sequence $(\phi_n)\subset \D$ such that $\w[\phi_n]\to 0$ and $\t[\phi_n-\phi_m]\to 0$ one has $\t[\phi_n]\to 0$;
	\item $\t$ is {\it $\w$-singular} (in symbols $\t\perp\w$)
	if for every $\phi\in \D$ there exists a sequence $(\phi_n)\subset\D$ verifying
	\begin{equation*}
	\lim_{n\to +\infty} \w[\phi_n]=0 \;\;\text{ and }\; \lim_{n\to +\infty} \t[\phi-\phi_n] =0,
	\end{equation*}
	or, equivalently, if for every $\psi\in \D$ there exists a sequence $(\psi_n)\subset\D$ verifying
	\begin{equation*}
	\lim_{n\to +\infty} \w[\psi-\psi_n]=0 \;\;\text{ and }\; \lim_{n\to +\infty} \t[\psi_n] =0,
	\end{equation*}
	(if $\t$ is non-negative, then it is $\w$-singular if and only if for every non-negative form $\mathfrak{p}$ with $\mathfrak{p} \leq \w$ and $\mathfrak{p} \leq \t$ one has $\mathfrak{p}=0$);
	\item $\t$ is {\it $\w$-regular} if there exists $\s\in M(\t)$ such that $\s\ll \w$;
	\item $\t$ is {\it $\w$-strongly singular} if there exists $\s\in M(\t)$ such that $\s\perp \w$.\\
\end{itemize}

The fundamental result in the theory of absolutely continuous and singular forms is the following decomposition (for the proof see \cite[Theorem 2.11]{HSdeS}, \cite[Theorem 2.3]{STT} or \cite[Corollary 4.5]{Tp_DB}).

\begin{theo}[Lebesgue decomposition of non-negative forms]
	\label{Leb_pos}
	Let $\s,\w$ be non-negative sesquilinear forms on $\D$. Then
	$$
	\s=\s_a+\s_s,
	$$
	where $\s_a$ and $\s_s$ are non-negative, $\w$-absolutely continuous and $\w$-singular forms, respectively. 
	Moreover, if $0\leq \mathfrak{u} \leq \s$ and $\mathfrak{u}$ is $\w$-absolutely continuous, then $\mathfrak{u}\leq \s_a$.
\end{theo}

\begin{rem}
	\label{rem_stong->sing}
	\begin{enumerate}
		\item[(i)] A simple class of $\w$-regular forms is the class of {\it $\w$-bounded} forms $\t$, verifying for some $C\geq 0$ the inequality $|\t(\phi,\psi)|\leq C \w[\phi]^\mez\w[\psi]^\mez$, for all $\phi,\psi\in \D$; i.e, $C\w\in M(\t)$.
		\item[(ii)] A non-negative $\w$-absolutely continuous form is $\w$-regular.
		\item[(iii)] A $\w$-strongly singular form $\t$ is $\w$-singular. Moreover, the converse holds if $\t$ is non-negative. 
	\end{enumerate}
\end{rem}

The $\w$-regularity in the non-negative case is weaker than the $\w$-absolute continuity as the next two examples show, in contrast with what stated in \cite[Proposition 4.8]{Tp_DB}.

\begin{exm}
	Let $\H$ be a Hilbert space with inner product $\pint$ and let $H$ be an unbounded positive self-adjoint operator with domain $D(H)$. Take $\kappa\notin D(H)$ and consider the projector $P\xi=\pin{\xi}{\kappa}\kappa$, $\xi\in \H$. 
	We indicate by $\w$, $\t$ and $\s$ the non-negative sesquilinear forms
	$$
	\w(\phi,\psi)=\pin{\phi}{\psi}, \qquad \t(\phi,\psi)=\pin{PH\phi}{H\psi}, \qquad \s(\phi,\psi)=\pin{H\phi}{H\psi}, 
	$$
	for  $\phi,\psi\in D(H)$, respectively. We have that $\s\ll \w$ and $\s\in M(\t)$, then $\t$ is $\w$-regular. Nevertheless, $\t$ is not $\w$-absolutely continuous. Indeed, were it so, then from
	$$
	\t[\phi]=\n{PH\phi}^2, \qquad \forall \phi \in D(H),
	$$
	$PH$ would be a closable operator in $\H$. But its adjoint $HP$ is not densely defined.
\end{exm}

As known, a non-negative form which is both $\w$-absolutely continuous and $\w$-singular is identically zero. The situation in our context is very different even in the non-negative case.

\begin{exm}
 Basing on \cite[Theorem 4.4]{HSdeS}, if $\s$ is a non-negative $\w$-absolutely continuous form but not $\w$-bounded, then there exists a non-negative $\w$-singular form $\t\neq 0$ such that $\t\leq \s$. This shows that there exist non-trivial (non-negative) forms $\t$ which are both $\w$-regular and $\w$-singular ($\w$-strongly singular). However, a particular case is given by the next proposition.
\end{exm}

\begin{pro}
	The only sesquilinear form which is $\w$-bounded and $\w$-singular is the null form. 
\end{pro}
\begin{proof}
	Let $\t$ be a $\w$-bounded and $\w$-singular sesquilinear form on $\D$. For every $\phi\in \D$ there exists a sequence $(\phi_n)\subset\D$ with the property that 
	\begin{equation*}
	\lim_{n\to +\infty} \w[\phi_n]=0 \;\;\text{ and }\; \lim_{n\to +\infty} \t[\phi-\phi_n] =0.
	\end{equation*}
	Note that, by the triangle inequality, $\{\w[\phi-\phi_n]\}$ is a bounded sequence. Therefore, for some $C\geq0$, 
	\begin{align*}
		|\t[\phi]|&\leq |\t(\phi_n,\phi)|+|\t(\phi-\phi_n,\phi_n)|+|\t[\phi-\phi_n]|\\
		&\leq C\w[\phi_n]^\mez\w[\phi]^\mez+C\w[\phi-\phi_n]^\mez\w[\phi_n]^\mez+|\t[\phi-\phi_n]|\to 0;
	\end{align*}
	i.e., $\t=0$.
\end{proof}

\no Two subsets of $\D$ related to a sesquilinear form $\t$ on $\D$ are
\begin{align*}
	K(\t)&=\{\phi\in \D:\t[\phi]=0\},\\
	\ker(\t)&=\{\phi\in \D: \t(\phi,\psi)=0, \forall \psi\in \D\}.
\end{align*}
In particular, the second one is a subspace of $\D$. Clearly, $\ker(\t)\subseteq K(\t)$ and the equality holds if $\t$ is non-negative by Cauchy-Schwarz inequality. Note that if $\t$ is not symmetric then $\ker(t)$ and $\ker(\t^*)$ may be different; however, we have also  $\ker(t^*)\subseteq K(t)$ and $K(\t)=K(\t^*)$.

There is a classical way to define a Hilbert space associated to a non-negative form $\w$ on $\D$. More precisely, the quotient $\D/\ker(\w)$ can be endowed with the inner product $\pin{\pi_\w(\phi)}{\pi_\w(\psi)}_\w:=\w(\phi,\psi)$, for all $\phi,\psi \in \D$, where $\pi_\w:\D\to \D/\ker(\w)$ is the canonical projection. The completion of $(\D/\ker(\w),\pint_\w)$ is denoted by $\H_\w$.

\begin{rem}
	\begin{enumerate}
		\item[(i)] If $\t$ is a $\w$-regular form, then $\ker(\w)\subseteq\ker(\t)$.
		\item[(ii)] Suppose that $\D$ has finite dimension. A form $\t$ is $\w$-regular if and only if $\t$ is $\w$-bounded if and only if $\ker(\w)\subseteq\ker(\t)$. By Remark \ref{rem_stong->sing} and the previous point, we have to prove only one implication. Namely, if $\ker(\w)\subseteq\ker(\t)$, then the form
\begin{equation*}
\overset{\sim}{\t}(\pi_{\w}(\phi),\pi_{\w}(\psi)):=\t(\phi,\psi), \qquad \pi_{\w}(\phi),\pi_{\w}(\psi)\in \D/\ker(\w),
\end{equation*}
is well-defined and therefore bounded by the norm of $\H_\w$; i.e., $\t$ is $\w$-bounded. 
	\end{enumerate}
\end{rem}

It is worth mentioning a characterization of non-negative singular forms involving the Hilbert spaces associated to them.

\begin{lem}[{\cite[Theorem 6.1]{Kos}}]
	\label{H_s_w_sing}
	A non-negative sesquilinear form $\s$ is $\w$-singular if and only if $\H_{\s+\w}$ is isomorphic to the cartesian product of $\H_\s$ and $\H_\w$
	($\H_{\s+\w}\simeq\H_\s \times \H_\w$).	
\end{lem}

We also recall that Theorem 3.6 of \cite{Tp_DB} gives a characterization of the $\w$-regular forms in terms of a representation in the space $\H_\w$. This expression is studied in another (but affine) context in \cite{Second} when $\w$ is the inner product of a Hilbert space.

\begin{exm}[{\cite[Remark 5.3]{Kos}}]
	\label{exm_ker_dens}
	Let $\t$ be a sesquilinear form on $\D$. If $\pi_\w(K(\t))$ is dense in $\H_\w$ (in particular, if $\pi_\w(\ker(\t))$ or $\pi_\w(\ker(\t^*))$ is dense in $\H_\w$), then $\t$ is trivially $\w$-singular. 
\end{exm}

\begin{rem}
	\label{cexm}
	One might ask if, in analogy to Theorem \ref{Leb_pos}, a sesquilinear form can be decomposed as a sum of a $\w$-regular form and a $\w$-strongly singular one. Here we prove that this is not allowed. 
	Indeed, consider $\D=\C^2$ and the sesquilinear forms given by
	\begin{align*}
		\t(\underline{x},\underline{y})&= x_1\ol{y_1}-x_2\ol{y_2}\\
		\w(\underline{x},\underline{y})&= x_1\ol{y_1}+x_1\ol{y_2}+x_2\ol{y_1}+x_2\ol{y_2}
	\end{align*}
	for all $\underline{x}:=(x_1,x_2),\underline{y}:=(y_1,y_2)\in \C^2$. Assume that 
	\begin{equation}
	\label{counter}
	\t=\t_{r}+\t_{ss}
	\end{equation}
	where $\t_{r}$ is a $\w$-regular form and $\t_{ss}$ is $\w$-strongly singular form. Then there exist two non-negative forms $\s_a$ and $\s_s$ such that $\s_a \ll \w$, $\s_s \perp \w$, $\s_a\in M(\t_r)$ and $\s_s\in M(\t_{ss})$.
	Since $\w[\underline{p}]=0$, where $\underline{p}=(1,-1)$, $\s_a[\underline{p}]=0$ and $\t_{r}[\underline{p}]=0$. 
	
	One has that $\t_{ss}=0$. In fact, it is clear if $\s_s=0$; on the other hand, if $\s_s\neq 0$
	by Lemma \ref{H_s_w_sing} there exists $\underline{q}\in \C^2$ for which $\C^2=\l\underline{p},\underline{q}\r$ and $\s_s[\underline{q}]=0$. This implies that $\t_{ss}(\underline{x},\underline{q})=0$ for all $\underline{x}\in \C^2$. Moreover, $0=\t[\underline{p}]=\t_r[\underline{p}]+\t_{ss}[\underline{p}]=\t_{ss}[\underline{p}]$. Therefore, $\t_{ss}=0$. 
	
	Hence, $|\t(\underline{x},\underline{p})|\leq \s_a[\underline{x}]^\mez \s_a[\underline{p}]^\mez=0$ for all $\underline{x}\in \C^2$. But this leads to a contradiction since $\t((1,1),\underline{p})\neq 0$. We conclude that (\ref{counter}) does not hold.
\end{rem}

In Theorem \ref{Leb_th} we will give a decomposition inspired to Theorem \ref{Leb_pos} involving one more type of form which is introduced by the next lemma.

\begin{lem}
\label{lem_def_mixed}
Let $\t$ be a sesquilinear form on $\D$. The following statements are equivalent.
	\begin{enumerate}
		\item There exist non-negative forms $\af,\bb$ such that $\af \ll \w$, $\bb \perp \w$, $\af \perp \bb$ and 
		\begin{equation}
		\label{def_car_mix}
		|\t[\phi]|\leq \af[\phi]^\mez\bb[\phi]^\mez, \qquad\forall \phi\in \D.
		\end{equation}
		\item There exist non-negative forms $\af,\bb$ such that $\af \ll \w$, $\bb \perp \w$, $\af \perp \bb$, $\af+\bb \in M(\t)$ and $\t[\phi]=0$ if $\af[\phi]=0$ or $\bb[\phi]=0$.
		\item There exist non-negative forms $\af,\bb$ such that $\af \ll \w$, $\bb \perp \w$, $\af \perp \bb$ and 
		$$
		|\t(\phi,\psi)|\leq \af[\phi]^\mez \bb[\psi]^\mez+\af[\psi]^\mez \bb[\phi]^\mez, \qquad \forall \phi,\psi\in \D.
		$$
		\item There exist forms $\t_1,\t_2,\af,\bb$ on $\D$ such that 
		$\t=\t_1+\t_2$, $\af,\bb$ are non-negative forms, $\af\ll\w$, $\bb\perp \w$, $\af \perp \bb$ and
		$$
		|\t_1(\phi,\psi)|\leq \af[\phi]^\mez \bb[\psi]^\mez,  \qquad	|\t_2(\phi,\psi)|\leq \af[\psi]^\mez \bb[\phi]^\mez, \qquad \forall \phi,\psi \in \D.
		$$
	\end{enumerate}
\end{lem}
\begin{proof}
(i) $\Rightarrow$ (ii) It is immediate.\\
(ii) $\Rightarrow$ (iii) Let us consider the bounded sesquilinear form
\begin{equation}
\label{eq_t_tilde_a_b}
\overset{\sim}{\t}(\pi_{\af+\bb}(\phi),\pi_{\af+\bb}(\psi)):=\t(\phi,\psi), \qquad \pi_{\af+\bb}(\phi),\pi_{\af+\bb}(\psi)\in \D/\ker(\af+\bb),
\end{equation}
and its closure $\overline{\t}$ on $\H_{\af+\bb}$. %We denote by $T$ the operator associated to $\overline{\t}$.\\ 
With similar meanings, we consider also the forms $\overline{\af}$ and $\overline{\bb}$. By Lemma \ref{H_s_w_sing}, $\H_{\af+\bb}$ can be decomposed as orthogonal sum of two subspaces,  $\H_{\af+\bb}=M_1\oplus M_2$, where $\overline{\af}$ is zero on $M_1$ and $\overline{\bb}$ is zero on $M_2$. 
Consequently, if $P$ is the orthogonal projector on $M_2$, the forms $\af$ and $\bb$ have the following expressions
$$
\af[\phi]=\n{P\pi_{\af+\bb}(\phi)}_{\af+\bb}^2, \qquad \bb[\phi]=\n{(I-P)\pi_{\af+\bb}(\phi)}_{\af+\bb}^2, \qquad \forall \phi \in \D.
$$
Since $\overline{\af}[P\pi_{\af+\bb}(\phi)]=0$ for all $\phi\in \D$, one has $\overline{\t}[P\pi_{\af+\bb}(\phi)]=0$ for all $\phi\in \D$ and, by the polarization identity, $\overline{\t}(P\pi_{\af+\bb}(\phi),P\pi_{\af+\bb}(\psi))=0$ for all $\phi,\psi\in \D$. In the same way, $\overline{\t}((I-P)\pi_{\af+\bb}(\phi),(I-P)\pi_{\af+\bb}(\psi))=0$ for all $\phi,\psi\in \D$. Hence, 
\begin{align*}
|\t(\phi,\psi)| &= |\overline{\t}(P\pi_{\af+\bb}(\phi),(I-P)\pi_{\af+\bb}(\psi))| \\
&+ |\overline{\t}((I-P)\pi_{\af+\bb}(\phi),P\pi_{\af+\bb}(\psi))|\\
&\leq \n{P\pi_{\af+\bb}(\phi)}_{\af+\bb}\n{(I-P)\pi_{\af+\bb}(\psi)}_{\af+\bb}\\
&+\n{P\pi_{\af+\bb}(\psi)}_{\af+\bb}\n{(I-P)\pi_{\af+\bb}(\phi)}_{\af+\bb}\\
&=\af[\phi]^\mez \bb[\psi]^\mez+\af[\psi]^\mez \bb[\phi]^\mez, \qquad\qquad \forall \phi,\psi \in \D.
\end{align*}
(iii) $\Rightarrow$ (iv) Clearly, $2(\af+\bb)\in M(\t)$. Following the proof of the previous part, the sesquilinear forms on $\D$ defined by
\begin{align*}
\t_1(\phi,\psi)&=\overline{\t}(P\pi_{\af+\bb}(\phi),(I-P)\pi_{\af+\bb}(\psi)), \\ \t_2(\phi,\psi)&=\overline{\t}((I-P)\pi_{\af+\bb}(\phi),P\pi_{\af+\bb}(\psi)), 
\end{align*}
satisfy the statement, up to rename $2\af$ and $2\bb$ with $\af$ and $\bb$, respectively.\\
(iv) $\Rightarrow$ (i) One obtains (\ref{def_car_mix}) replacing $\af$ with $2\af$ and $\bb$ with $2\bb$, which are still $\w$-absolutely continuous and $\w$-singular, respectively, and singular with respect to each other.
\end{proof}

\begin{defin}
A sesquilinear form is said {\it $\w$-mixed} if it satisfies one of the statements in Lemma \ref{lem_def_mixed}.
\end{defin}

We now conclude this section by giving some examples.

\begin{exm}
	It is easy to see, using Lemma \ref{lem_def_mixed}(ii), that the form $\t$ of Remark \ref{cexm} is $\w$-mixed, taking $\af=\w$ and $\bb$ defined by $\bb(\underline{x},\underline{y})= x_1\ol{y_1}-x_1\ol{y_2}-x_2\ol{y_1}+x_2\ol{y_2}$, for all $\underline{x},\underline{y}\in \C^2$. However, $\t$ is also $\w$-singular. Indeed, for $\phi=(x_1,x_2)$ the constant sequence $\phi_n:=\mez (x_1+x_2,x_1+x_2)$ satisfy $\t[\phi_n]=0$ and $\w[\phi-\phi_n]=0$. This fact and Remark \ref{cexm} show that there exist $\w$-singular forms which are not $\w$-strongly singular.	
\end{exm}

\begin{exm}
	Let $H$ be a self-adjoint operator with domain $D(H)$ on a Hilbert space $(\H,\pint)$. Define two sesquilinear form on $\D:=D(H)\times D(H)$ as 
	\begin{align*}
		\w(\underline{\xi}, \underline{\eta})=\pin{\xi_1}{\eta_1}, \qquad
		\t(\underline{\xi}, \underline{\eta})=\pin{H\xi_1}{\eta_2}+\pin{H\xi_2}{\eta_1},
	\end{align*}
	for $\underline{\xi}=(\xi_1,\xi_2),\underline{\eta}=(\eta_1,\eta_2)\in \D$. It is easy to check that $\t$ satisfies (\ref{def_car_mix}) with 
	$$
	\af(\underline{\xi}, \underline{\eta})=\pin{H\xi_1}{H\eta_1},
	\qquad 
	\bb(\underline{\xi}, \underline{\eta})=\pin{\xi_2}{\eta_2}, \qquad  \underline{\xi},\underline{\eta}\in \D.
	$$
\end{exm}

\begin{exm}
	Let $\D:=C(0,1)$ stand for the vector space of continuous functions on the interval $[0,1]$. It is well-known that the non-negative forms
	$$
	\w(f,g)=\int_0^1 f(x)\ol{g(x)}dx, \qquad \bb(f,g)=f(0)\ol{g(0)}, \qquad f,g\in \D,
	$$
	are singular with respect to each other (in particular, $\bb$ is a form of the type of Example \ref{exm_ker_dens}). Consequently, the sesquilinear form
	$$
	\t(f,g)=f(0)\int_0^1 \ol{g(x)}dx, \qquad f,g\in \D,
	$$
	is $\w$-mixed.
\end{exm}

\begin{exm}
	Let $\H_- \supset \H \supset \H_+$ be a rigged Hilbert space with duality $\pint$ between $\H_-$ and $\H_+$. Given $\omega,\varrho\in \H_-$ we define the sesquilinear form
	$$
	\t(\xi,\eta)=\pin{\omega}{\xi}\ol{\pin{\varrho}{\eta}}, \qquad \xi,\eta \in \H_+,
	$$
	and let $\w(\xi,\eta)=\pin{\xi}{\eta}$ for $\xi,\eta \in \H_+$. Taking into account \cite[Examples 1.15, 5.5, 5.9]{Kos}, we can state that
	\begin{itemize}
		\item if $\omega,\varrho\in \H$, then $\t$ is $\w$-bounded;
		\item if $\omega\in\H_-\backslash\H,\varrho\in \H$ or $\varrho\in \H_-\backslash\H,\omega\in \H$, then $\t$ is $\w$-mixed;
		\item if $\omega$ or $\rho$ is in $\H_-\backslash\H$, then $\ker(\t)$ is dense in $\H$ and therefore $\t$ is $\w$-singular;
		\item if $V\cap \H=\{0\}$, where $V$ is the subspace of $\H_-$ generated by $\omega$ and $\varrho$, then $\t$ is $\w$-strongly singular.
	\end{itemize}
\end{exm}

In the rest of this section we analyze the definitions given at the beginning in some special cases. We recall that in our approach a form is not in general non-negative; however forms with a restricted set of values can have a interest (see Proposition \ref{pro_N_t} below).  We start with the following relations between a form, its adjoint, the real and the imaginary parts, which are easy to prove.

\begin{pro}
	\label{pro_re_im}
	Let $\t$ be a sesquilinear form on $\D$.
	\begin{enumerate}
		\item The sets $M(\t)$ and $M(\t^*)$ are equal. Furthermore, 
		$$
		M(\Re\t)+M(\Im\t)\subseteq M(\t)\subseteq M(\Re\t)\cap M(\Im\t),
		$$
		where $M(\Re\t)+M(\Im\t):=\{\s_1+\s_2:\s_1\in M(\Re\t),\s_2\in M(\Im\t)\}$.
		\item If $\t$ is $\w$-regular ($\w$-singular, $\w$-strongly singular or $\w$-mixed), then the same holds for $\t^*$, $\Re \t$ and $\Im \t$.
	\end{enumerate}
\end{pro}

\no We denote by  $N(\t)$ the {\it positively homogeneous} subset of $\C$
$$
N(\t) :=\{\t[\phi]:\phi\in \D\}.
$$
Positively homogeneous means that $\alpha N(\t)=N(\t)$ for all $\alpha >0$. By definition, $\t$ is non-negative if and only if $N(\t) = [0,+\infty)$. Moreover, $\t$ is symmetric if and only if $N(\t)\subseteq \R$.

\begin{rem}
If $\D$ is a subspace of a Hilbert space with norm $\nor$, then a more important (convex) set is the so-called {\it numerical range} (see \cite[Chapter VI]{Kato} and \cite{Halmos_nr,Schm} for the operator case) defined by 
$$
\mathfrak{N}(\t):=\{\t[\phi]:\phi\in \D, \n{\phi}=1\}.
$$
Clearly, $\mathfrak{N}(\t) \subseteq N(\t)$ and $N(t)$ is contained in one of the following subsets of $\C$
\begin{align*}
[0,+\infty),\qquad\qquad\;\;\; \R,\qquad\qquad\;\; \mathcal{Q}:=\{\lambda \in \C:\Re \lambda \geq 0, \Im \lambda \geq 0\},\\
\Pi:=\{\lambda \in \C:\Re \lambda \geq 0\},  \qquad
\mathcal{S}_c:=\{\lambda \in \C:|\Im \lambda|\leq c\Re \lambda\} \;\;(c\geq 0), 
\end{align*}
if and only if $\mathfrak{N}(\t)$ is contained in the same one. We mention that the last subset (a {\it sector} of $\C$) above plays a special role in the theory of representation by a linear operator of a sesquilinear form (see \cite[Chapter VI]{Kato} and \cite{RC_CT,Second} for generalizations).
\end{rem}

For forms $\t$ with special set $N(\t)$ the notions introduced in the previous section are simplified.

\begin{pro}
	\label{pro_N_t}
	Let $\t$ be a sesquilinear form on $\D$. The following statements hold.
	\begin{enumerate}
		\item If $\t$ is non-negative and $\w$-mixed, then $\t=0$. 
		\item Assume that $N(\t)\subseteq \mathcal{Q}$. Then
			\begin{enumerate}
				\item[\emph{(a)}] $2(\Re\t+\Im\t)\in M(\t)$;
				\item[\emph{(b)}] $\t$ is $\w$-singular if and only if $\t$ is $\w$-strongly singular if and only if $\Re\t+\Im\t$ is $\w$-singular;
				\item[\emph{(c)}] if $\t$ is $\w$-mixed, then $\t=0$.
			\end{enumerate} 		
		\item Assume that $N(\t)\subseteq \mathcal{S}_c$, with $c\geq0$. Then
		\begin{enumerate}
			\item[\emph{(a)}] $(1+c)\Re\t\in M(\t)$;
			\item[\emph{(b)}] $\t$ is $\w$-singular if and only if $\t$ is $\w$-strongly singular if and only if $\Re\t$ is $\w$-singular;
			\item[\emph{(c)}] if $\t$ is $\w$-mixed, then $\t=0$.
		\end{enumerate}  
		\item If $N(\t)\subseteq\Pi$ and $\t$ is $\w$-mixed, then $\Re\t=0$.
	\end{enumerate}
\end{pro}
\begin{proof}
	\begin{enumerate}
		\item[(i)] Assume that (\ref{def_car_mix}) holds and adopt the notation of the proof of Lemma \ref{lem_def_mixed}. The space $\H_{\af+\bb}$ is the orthogonal sum of two subspaces $M_1$ and $M_2$ where $\overline{\af}$ is zero on $M_1$ and $\overline{\bb}$ is zero on $M_2$. Moreover let $\overline{\t}$ be closure of the form in (\ref{eq_t_tilde_a_b}). By (\ref{def_car_mix}) $\overline{\t}$ vanishes on $M_1$ and on $M_2$; hence $\overline{\t}=0$ on $\H_{\af+\bb}$, because of the Cauchy-Schwarz inequality.	
		\item[(ii)] For (a) we have 
		\begin{align*}
		\label{eq_|Q|}
		|\t(\phi,\psi)|&\leq |\Re\t(\phi,\psi)|+|\Im\t(\phi,\psi)| \\
		&\leq \Re\t[\phi]^\mez\Re\t[\psi]^\mez+\Im\t[\phi]^\mez\Im\t[\psi]^\mez \nonumber \\
		&\leq 2(\Re\t+\Im \t)[\phi]^\mez (\Re\t+\Im \t)[\psi]^\mez,  \nonumber \qquad \forall \phi,\psi\in \D.
		\end{align*} 
		To prove (b) we notice that if $\t$ is $\w$-singular, then so $\Re \t+ \Im \t$ is, because $|\t[\phi]|^2=\Re \t[\phi]^2+\Im\t[\phi]^2$. The singularity of $\Re \t+ \Im \t$ implies that $\t$ is $\w$-strongly singular.   \\
		For proving (c) assume that $\t$ is $\w$-mixed. Proposition \ref{pro_re_im} implies that $\Re \t, \Im \t$ are $\w$-mixed. Since $\Re \t, \Im \t \geq 0$, by the previous case, $\t=0$. The last implication we need is given by Remark \ref{rem_stong->sing}.
		\item[(iii)] Similar considerations as above apply to this statement.
		\item[(iv)] In this case $\Re\t\geq 0$ and $\w$-mixed. Therefore, $\Re\t=0$ by point (i). \qedhere
	\end{enumerate}
\end{proof}

\section{Lebesgue decomposition theorem}
\label{sec_Leb}

Now, we prove the main theorem of this paper, whose proof is based on Theorem 4.3 of \cite{Tp_DB}. To do this we will use the following construction of \cite{STT} of the $\w$-absolutely continuous $\s_a$ and $\w$-singular $\s_s$ parts of a non-negative form $\s$.\\
Let $J$ be the embedding operator $\pi_{\s+\w}(\phi)\to \pi_w (\phi)$, from $\D/\ker{(\s+\w)}\sub \H_{\s+\w}$ into $\H_\w$. In particular, $J$ is a densely defined contraction and  $J^{**}$ is the closure of $J$. 
If $P$ is the orthogonal projection of $\H_{s+w}$ onto $\{\ker J^{**}\}^\perp$, then for all $\phi,\psi \in \D$,
\begin{align*}
%\label{eq_sa+w}
(\s_a+\w)(\phi,\psi)&=\pin{P\pi_{\s+\w}(\phi)}{\pi_{\s+\w}(\psi)}_{\s+\w} \\
\s_s(\phi,\psi)&=\pin{(I-P)\pi_{\s+\w}(\phi)}{\pi_{\s+\w}(\psi)}_{\s+\w} \nonumber.
\end{align*}

\no We stress that $\s_a+\w$ and $\s_s$ are also singular with respect to each other.

\begin{theo}
	\label{Leb_th}
	Let $\t,\w$ be sesquilinear forms on $\D$, with $\w$ non-negative and $M(\t)\neq \varnothing$. Then, for any $\s\in M(\t)$,
	$$
	\t=\t_r+\t_m+\t_{ss},
	$$
	where $\t_r$ is a $\w$-regular form, $\t_m$ is $\w$-mixed form and $\t_{ss}$ is a $\w$-strongly singular form on $\D$.
	%$\t_r \ll \w$ and $\t_s\perp\w$.
\end{theo}
\begin{proof}
	Take $\s\in M(\t)$.
	A well-defined bounded sesquilinear form on $\D/\ker(\s+\w)$ can be defined as
	\begin{equation*}
	\label{t_tilde}
	\overset{\sim}{\t}(\pi_{\s+\w}(\phi),\pi_{\s+\w}(\psi)):=\t(\phi,\psi), \qquad \forall \pi_{\s+\w}(\phi),\pi_{\s+\w}(\psi)\in \D/\ker(\s+\w).
	\end{equation*}
	There exists a unique bounded operator $T$ on $\H_{s+w}$, whose norm is not greater than $1$, such that
	\begin{equation*}
	\label{eq_t_T}
	\t(\phi,\psi)=\pin{T\pi_{\s+\w}(\phi)}{\pi_{\s+\w}(\psi)}_{\s+\w}, \qquad \forall \phi,\psi \in \D.
	\end{equation*}
	Set 
	\begin{align}
	\label{t_r} \t_r(\phi,\psi)&:=\pin{TP\pi_{\s+\w}(\phi)}{P\pi_{\s+\w}(\psi)}_{\s+\w}\\
	%\label{t_p_sing}
	\t_m(\phi,\psi)&:=\pin{TP\pi_{\s+\w}(\phi)}{(I-P)\pi_{\s+\w}(\psi)}_{\s+\w} \nonumber \\
	&+ \pin{T(I-P)\pi_{\s+\w}(\phi)}{P\pi_{\s+\w}(\psi)}_{\s+\w} \nonumber\\
	%\label{t_s_sing}
	\t_{ss}(\phi,\psi)&:=\pin{T(I-P)\pi_{\s+\w}(\phi)}{(I-P)\pi_{\s+\w}(\psi)}_{\s+\w} \nonumber
	\end{align}
	for all $\phi,\psi \in \D$. We have $\t=\t_r+\t_m+\t_{ss}$. In addition, $\t_r$ is $\w$-regular, $\t_m$ is $\w$-mixed and $\t_{ss}$ is $\w$-strongly singular. In fact, for all $\phi,\psi \in \D$, 
	\begin{align*}
	|\t_r(\phi,\psi)|&\leq \nosw{T}\n{P\pi_{\s+\w}(\phi)}_{\s+\w} \n{P\pi_{\s+\w}(\psi)}_{\s+\w}	\\
	&= (\s_a+\w)[\phi]^\mez(\s_a+\w)[\psi]^\mez;\\
	|\t_m(\phi,\psi)|&\leq \nosw{T}\n{P\pi_{\s+\w}(\phi)}_{\s+\w} \n{(I-P)\pi_{\s+\w}(\psi)}_{\s+\w}\\
	&+
	\nosw{T}\n{(I-P)\pi_{\s+\w}(\phi)}_{\s+\w} \n{P\pi_{\s+\w}(\psi)}_{\s+\w}	\\
	% & \leq \n{P\pi_{\s+\w}(\phi)}_{\s+\w} \n{P\pi_{\s+\w}(\psi)}_{\s+\w} \\
	&\leq  (\s_a+\w)[\phi]^\mez\s_s[\psi]^\mez + (\s_a+\w)[\psi]^\mez \s_s[\phi]^\mez;\\
	|\t_{ss}(\phi,\psi)|&\leq \nosw{T}\n{(I-P)\pi_{\s+\w}(\phi)}_{\s+\w} \n{(I-P)\pi_{\s+\w}(\psi)}_{\s+\w}\\
	&\leq  \s_s[\phi]^\mez\s_s[\psi]^\mez.     \qedhere
	\end{align*}
\end{proof}

\begin{rem}
	The sesquilinear form $\t_s:=\t_m+\t_{ss}$ is the $\w$-singular part of $\t$ according to \cite[Theorem 4.3]{Tp_DB}. To prove that $\t_s$ is actually $\w$-singular, let $\phi \in \D$ and $(\phi_n)\subset \D$ such that $\pi_{\s+\w}(\phi_n)\to (I-P)\pi_{\s+\w}(\phi)$. Therefore, $\w[\phi_n]\leq (\w+\s_a)[\phi_n]\to 0$ and $\t_s[\phi-\phi_n]\to 0$.
\end{rem}

\begin{rem}
	The decomposition in Theorem \ref{Leb_pos} is a special case of Theorem \ref{Leb_th} taking $\s=\t$. In particular, with the notations of these theorems, $\t_{r}=\t_a$, $\t_m=0$ and $\t_{ss}=\t_s$.
\end{rem}

\begin{rem}
	The decomposition of a form $\t$ into a sum of $\w$-regular, $\w$-mixed and $\w$-strongly singular parts is not unique, even if $\t$ is non-negative, as it is well-known (see \cite[Theorem 4.6]{HSdeS}). In addition, the particular decomposition given by Theorem \ref{Leb_th} depends also on the choice of $\s\in M(\t)$ as we show here (we will follow the construction of the proof above).
	
	Set $\D=\C^3$. We indicate by $e_1,e_2,e_3$ the vectors $(1,0,0),(0,1,0),(0,0,1)$, respectively. Here, for convenience, we represent all sesquilinear forms by their associated matrices with respect to the basis $\{e_1,e_2,e_3\}$. Consider the sesquilinear forms $\t,\s,\w$ on $\C^3$ which are represented by the following matrices 
	$$
	\begin{pmatrix}
	-1 & 0 & 0  \\
	0 & 1 & 0  \\
	0 & 0 & 0  
	\end{pmatrix}
	,\qquad
	\begin{pmatrix}
	1 & 0 & 0  \\
	0 & 1 & 0  \\
	0 & 0 & 0  
	\end{pmatrix}
	,\qquad
	\begin{pmatrix}
	0 & 0 & 0  \\
	0 & 1 & 0  \\
	0 & 0 & 1  
	\end{pmatrix}
	$$
	Clearly $\s\in M(\t)$ and $J^{**}=J$ is defined as	
    $J^{**}:\H_{\s+\w}  \to  \C^3/ \ker(\w)$, 	$J^{**}:\phi   \mapsto  \phi +\text{span} \{e_1\},$
	where $\H_{\s+\w}$ is the space $\C^3$ with the norm $\nor_{\s+\w}$. Moreover, $\ker J^{**}=\text{span} \{e_1\}$, $\{\ker J^{**}\}^\perp=\text{span}\{ e_2,e_3 \}$ and the projector $P$ is defined as $P(\phi_1,\phi_2,\phi_3)  = (0,\phi_2,\phi_3)$.
	The Lebesgue decomposition $\s=\s_a+\s_s$ of $\s$ with respect to $\w$ is then 
	$$
	\begin{pmatrix}
	1 & 0 & 0  \\
	0 & 1 & 0  \\
	0 & 0 & 0  
	\end{pmatrix} =\begin{pmatrix}
	0 & 0 & 0  \\
	0 & 1 & 0  \\
	0 & 0 & 0  
	\end{pmatrix}+
	\begin{pmatrix}
	1 & 0 & 0  \\
	0 & 0 & 0  \\
	0 & 0 & 0  
	\end{pmatrix}.
	$$
	Note that $\t(\phi,\psi)=\pin{T\phi}{\psi}_{\s+\w},$ for all $\phi,\psi \in \C^3$, where  $T(\phi_1,\phi_2,\phi_3)=(-\phi_1,\mez\phi_2,0)$. With this we recover that the Lebesgue decomposition $\t=\t_r+\t_m+\t_{ss}$ of $\t$ with respect to $\w$ and taking $\s\in M(\t)$ is
	$$
	\begin{pmatrix}
	-1 & 0 & 0  \\
	0 & 1 & 0  \\
	0 & 0 & 0  
	\end{pmatrix}=
	\begin{pmatrix}
	0 & 0 & 0  \\
	0 & 1 & 0  \\
	0 & 0 & 0  
	\end{pmatrix}+
		\begin{pmatrix}
	0 & 0 & 0  \\
	0 & 0 & 0  \\
	0 & 0 & 0  
	\end{pmatrix}+
	\begin{pmatrix}
	-1 & 0 & 0  \\
	0 & 0 & 0  \\
	0 & 0 & 0  
	\end{pmatrix}.
	$$
	Now, let $\mathfrak{u}$  the non-negative sesquilinear form which corresponds to the matrix
	$$
	\begin{pmatrix}
	\frac{5}{3} & -\frac{4}{3}  & 0  \\
	-\frac{4}{3}  & \frac{5}{3} & 0  \\
	0 & 0 & 0 
	\end{pmatrix}.
	$$
	We have that $\mathfrak{u}-\t$ and $\mathfrak{u}+\t$ are non-negative forms, then $\mathfrak{u}\in M(\t)$.
	Therefore, $\H_{\mathfrak{u}+\w}$ is $\C^3$ with the norm $\nor_{\mathfrak{u}+\w}$, the new operator $J^{**}$ is defined as before and $\ker J^{**}=\text{span} \{e_1\}$. But now $\{\ker J^{**}\}^\perp=\text{span} \{(4,5,0),(0,0,1)\}$ and the projection $P_\mathfrak{u}$ on $\{\ker J^{**}\}^\perp$ is $P_\mathfrak{u}(\phi_1,\phi_2,\phi_3)  = (\frac{4}{5}\phi_2,\phi_2,\phi_3)$.
	The Lebesgue decomposition $\mathfrak{u}=\mathfrak{u}_a+\mathfrak{u}_s$ of $\s$ with respect to $\w$ is
	$$
	\begin{pmatrix}
	\frac{5}{3} & -\frac{4}{3}  & 0  \\
	-\frac{4}{3}  & \frac{5}{3} & 0  \\
	0 & 0 & 0 
	\end{pmatrix}=
	\begin{pmatrix}
	0 & 0 & 0  \\
	0 & \frac{3}{5} & 0  \\
	0 & 0 & 0  
	\end{pmatrix}+
	\begin{pmatrix}
	\frac{5}{3} & -\frac{4}{3}  & 0  \\
	-\frac{4}{3}  & \frac{16}{15} & 0  \\
	0 & 0 & 0 
	\end{pmatrix}
	$$
	Moreover, 
	$\t(\phi,\psi)=\pin{T_{\mathfrak{u}}\phi}{\psi}_{\mathfrak{u}+\w},$ for all $\phi,\psi \in \C^3$, where  $T_\mathfrak{u}(\phi_1,\phi_2,\phi_3)=(-\phi_1-\frac{1}{2}\phi_2,\frac{1}{2}\phi_1+\frac{5}{8}\phi_2,0)$ and, finally, the Lebesgue decomposition $\t=\t_r'+\t_m'+\t_{ss}'$ of $\t$ with respect $\w$ and taking $\mathfrak{u}\in M(\t)$
	$$
	\begin{pmatrix}
	-1 & 0 & 0  \\
	0 & 1 & 0  \\
	0 & 0 & 0  
	\end{pmatrix}=
	\begin{pmatrix}
	0 & 0 & 0  \\
	0 & \frac{9}{25} & 0  \\
	0 & 0 & 0  
	\end{pmatrix}+
		\begin{pmatrix}
	0 & -\frac{4}{5} & 0  \\
	-\frac{4}{5} & \frac{32}{25} & 0  \\
	0 & 0 & 0  
	\end{pmatrix}+
	\begin{pmatrix}
	-1 & \frac{4}{5} & 0  \\
	\frac{4}{5} & -\frac{16}{25} & 0  \\
	0 & 0 & 0  
	\end{pmatrix}.
	$$
	
	In the rest of the paper, we refer to Theorem \ref{Leb_th} as the {\it Lebesgue decomposition} of a form $\t$ with respect to $\w$ and $\s\in M(\t)$.
\end{rem}

\begin{pro}
	Let $\t=\t_r+\t_m+\t_{ss}$ be the Lebesgue decomposition of a sesquilinear form $\t$ with respect to $\w$ and $\s\in M(\t)$. 
	\begin{enumerate}
		\item The Lebesgue decomposition with respect to $\w$ and $\s$ of $\t^{*}, \Re \t$ and $\Im \t$ are
		\begin{align*}
			\t&=(\t_r)^*+(\t_m)^*+(\t_{ss})^*,\\
			\Re\t&=\Re(\t_r)+\Re(\t_m)+\Re(\t_{ss}),\\
			\Im\t&=\Im(\t_r)+\Im(\t_m)+\Im(\t_{ss}),
		\end{align*}
		respectively. In particular, if $\t$ is symmetric, then $\t_r,\t_m$ and $\t_{ss}$ are symmetric.
		\item The sets $N({\t_{r}}),N({\t_{ss}})$ are contained in $N(\t)$. In particular, if $\t$ is non-negative, then $\t_r$ and $\t_{ss}$ are non-negative.
	\end{enumerate}
\end{pro}

The $\w$-mixed part is not in general non-negative (and consequently the null form by Proposition \ref{pro_N_t}) if $\t$ is non-negative. For instance, one can take $\w$-mixed part of $\t$ with respect to $\w$ and $\s\in M(\t)$, where $\t,\s,\w$ are represented by the matrices
		$$
	\begin{pmatrix}
	2 & 1 & 0  \\
	1 & 2 & 0  \\
	0 & 0 & 0  
	\end{pmatrix}
	,\qquad
	\begin{pmatrix}
	3 & 0 & 0  \\
	0 & 3 & 0  \\
	0 & 0 & 0  
	\end{pmatrix}
	,\qquad
	\begin{pmatrix}
	0 & 0 & 0  \\
	0 & 1 & 0  \\
	0 & 0 & 1  
	\end{pmatrix},
	$$
	respectively.

\section{Measures and sesquilinear forms}
\label{sec_meas}

In this section we show that one can prove the Lebesgue decomposition of (complex) measures with the help of Theorem \ref{Leb_th}. We refer to \cite{Rudin}  for the notions and results of the Measure Theory (see also \cite{Halmos_m,Rao}). All the measures that we will consider are finite. 

Let $\RR$ stand for a $\sigma$-algebra on a non-empty set $\mathcal{A}$. We write $\D:=S(\mathcal{A},\RR)$ for the complex vector space of simple functions on $(\mathcal{A},\RR)$. Let $\mu$ be a (complex) measure on $(\mathcal{A},\RR)$. We said that $\mu$ is
\begin{itemize}
	\item {\it signed} if $\mu(A)\in \R$ for all $A\in \RR$;
	\item {\it non-negative} if $\mu(A)\geq 0$ for all $A\in \RR$.		
\end{itemize}

The {\it total variation} of a measure $\mu$ is the non-negative measure $|\mu|$ on $(\mathcal{A},\RR)$ defined on $A\in \RR$ as
$$
|\mu|(A):=\sup \sum_{k=1}^{\infty} |\mu(A_k)|,
$$
where the supremum is taken over all sequences $\{A_k\}$ of disjoint subsets in $\RR$ such that $\bigcup_k A_k=A$. The importance of $|\mu|$ is that it is the smaller non-negative measure $\kappa$ that {\it bounds} $\mu$; i.e.,  $|\mu(A)|\leq \kappa(A)$ for all $A\in \RR$.\\
The {\it characteristic function} of a subset $A\in \RR$ will be indicated by $\chi_A$.

Given two measures $\mu,\nu$ on $(\mathcal{A},\RR)$ with $\nu$ non-negative, $\mu$ is {\it $\nu$-absolutely continuous}  (in symbol $\mu \ll \nu$) if the following equivalent conditions are satisfied:
\begin{enumerate}
	\item[(a1)] if $\nu(A)=0$ implies $\mu(A)=0$;
	\item[(a2)] for every $\epsilon>0$ there exists $\delta>0$ such that $|\mu(A)|<\epsilon$ for all $A\in \RR$ with $\nu(A)<\delta$, or equivalently in a different notation, $\displaystyle \lim_{\nu(A)\to 0} \mu(A) =0$.
\end{enumerate}

On the other hand, $\mu$ is {\it $\nu$-singular}  (in symbol $\mu \perp \nu$) if  one of the following equivalent conditions is satisfied (see \cite[Theorem 6.1.17]{Rao})
\begin{enumerate}
	\item[(s1)] there exists $E\in \RR$ such that $\nu(A)=\nu(A\cap E)$ and $\mu(A)=\mu(A\cap E^c)$;
	\item[(s2)] $\forall \epsilon>0$ there exists $E_\epsilon\in \RR\text{ such that } \mu_s(E_\epsilon)<\epsilon \text{ and }\nu(\mathcal{A}\backslash E_\epsilon)<\epsilon$.
\end{enumerate}

Furthermore, $\mu$ is $\nu$-absolutely continuous (resp. $\nu$-singular) if and only if $|\mu|$ is $\nu$-absolutely continuous (resp. $\nu$-singular) if and only if there exists an $\nu$-absolutely continuous (resp. $\nu$-singular) non-negative measure $\tau$ on $(\mathcal{A},\RR)$ bounding $\mu$. 

A sesquilinear form $\t$ on $\D=S(\mathcal{A},\RR)$ is said to be {\it induced} by the measure $\mu$ on $(\mathcal{A},\RR)$ if
$$
\t(\phi,\psi)=\int_\mathcal{A} \phi \ol{\psi}d\mu, \qquad \forall \phi,\psi\in \D.
$$
Let $\mu,\nu$ be two measures on $(\mathcal{A},\RR)$ with $\nu$ non-negative. Consider the sesquilinear forms induced by $\mu,|\mu|$ and $\nu$; i.e.,
\begin{equation}
\label{form_meas}
 \t(\phi,\psi)=\int_\mathcal{A} \phi \ol{\psi}d\mu, \qquad \s(\phi,\psi)=\int_\mathcal{A} \phi \ol{\psi}d|\mu|, \qquad
 \w(\phi,\psi)=\int_\mathcal{A} \phi \ol{\psi}d\nu, 
\end{equation}
for all $\phi,\psi\in \D=S(\mathcal{A},\RR)$, respectively. Obviously, $\s\in M(\t)$ and $\t$ is non-negative (resp. symmetric) if and only if $\mu$ is non-negative (resp. signed).

\begin{lem}
	\label{lem_reg_for_meas}
	The following statements hold.
	\begin{enumerate}
		\item The form $\t$ is $\w$-regular if and only if $\mu$ is $\nu$-absolutely continuous.
		\item If $\mu$ is $\nu$-singular, then $\mu$ is $\w$-strongly singular. The converse is true if $\t$ is non-negative.
		\item If $\s$ is $\w$-singular, then $\mu$ is $\nu$-singular.
	\end{enumerate} 
\end{lem}
\begin{proof}
	\begin{enumerate}
		\item[(i)] Assume $\t$ is $\w$-regular. By definition, there exists $\mathfrak{u}\in M(\t)$ and $\mathfrak{u} \ll \w$. If $A\in \RR$ and $\nu(A)=0$ then $\chi_A\in \ker \w \sub \ker \mathfrak{u} \sub \ker \t$. Therefore, $\mu(A)=0$.	Conversely, if $\mu$ is $\nu$-absolutely continuous, then so $|\mu|$ is and $\s\ll\w$ by \cite[Theorem 3.2]{STT}. Since $\s\in M(\t)$, $\t$ is $\w$-regular.
		\item[(ii)] In \cite[Theorem 3.2]{STT} it was proved that if $\mu$ is non-negative, then $\t$ is $\w$-singular if and only if $\mu$ is $\nu$-singular. In the general case, assume that $\mu$ is $\nu$-singular. This means that $|\mu|$ is $\nu$-singular and, consequently, $\s\perp \w$. Finally, $\s\in M(\t)$ implies that $\t$ is $\w$-strongly singular.
		\item[(iii)] If $\s$ is $\w$-singular, then $|\mu|$ is $\nu$-singular. Hence, $\mu$ is $\nu$-singular. \qedhere
	\end{enumerate}
\end{proof}

Now we can give the announced proof of the Lebesgue decomposition theorem of finite measures based on the ideas developed in this paper. We state it for reader's convenience. 

\begin{theo}
	Let $\RR$ be a $\sigma$-algebra on a non-empty set $\mathcal{A}$. Let $\nu,\mu$ be measures on $(\mathcal{A},\RR)$, $\nu$ being non-negative.
	There exist unique measures $\mu_a,\mu_s$ on $(\mathcal{A},\RR)$ such that
	\begin{enumerate}
		\item $\mu=\mu_a+\mu_s$;
		\item $\mu_a$ is $\nu$-absolutely continuous and $\mu_s$ is $\nu$-singular.
	\end{enumerate}
\end{theo}
\begin{proof}
	The uniqueness follows easily by the following argument. Indeed, assume that $\mu=\mu_a + \mu_s=\mu_a' + \mu_s'$, where $\mu_a,\mu_a'\ll \nu$ and $\mu_s,\mu_s '\perp \nu$. Then $\mu_a-\mu_a'=\mu_s'-\mu_s$; i.e., $\mu_a-\mu_a'$ is both absolutely continuous and singular with respect to $\nu$. %By \cite[pag. 131]{DS},
	Thus, clearly,  $\mu_a=\mu_a'$ and $\mu_s=\mu_s'$.
	
	To prove the existences, let us define the forms $\t,\w,\s$ as in (\ref{form_meas}).  
	First of all, assume that $\mu$ is non-negative; i.e.,  $\t=\s\geq 0$. 
	Consider the Lebesgue decomposition $\s=\s_a+\s_s$ of $\s$ with respect to $\w$ as in Theorem \ref{Leb_pos}. 
	Moreover, with the notations introduced before Theorem \ref{Leb_th}, for all $\phi,\psi\in \D$,
	\begin{align*}
	(\s_a+\w)(\phi,\psi)&=\pin{P\pi_{\s+\w}(\phi)}{P\pi_{\s+\w}(\psi)}_{\s+\w},\\
	\s_s(\phi,\psi)&=\pin{(I-P)\pi_{\s+\w}(\phi)}{(I-P)\pi_{\s+\w}(\psi)}_{\s+\w}.
	\end{align*}
	
	We know (see \cite[Theorem 3.4]{STT}%\cite[Theorem 5.5]{HSdeS}
	) that there exist additive set functions $\mu_a$ and $\mu_s$ on $(\mathcal{A},\RR)$ satisfying $\mu=\mu_a+\mu_s$ and
	\begin{align*}
	(\s_a+\w)(\phi,\psi)&=\int_\mathcal{A} \phi \ol{\psi}d(\mu_a+\nu), \\
	\s_s(\phi,\psi)&=\int_\mathcal{A} \phi \ol{\psi}d\mu_s, \qquad \forall \phi,\psi\in \D.
	\end{align*}
	In addition, $\mu_a$ is a $\nu$-absolutely continuous and $\mu_s$ is a $\nu$-singular in the sense of \cite[Section 3]{STT}; i.e, $\displaystyle \lim_{\nu(A)\to 0} \mu_a(A)=0$ 
	and
	$$
	\forall\epsilon>0\; \exists E_\epsilon\in \RR\text{ such that } \mu_s(E_\epsilon)<\epsilon \text{ and }\nu(\mathcal{A}\backslash E_\epsilon)<\epsilon.
	$$
	We prove that $\mu_a,\mu_s$ are continuous from below; then they must be measures (see \cite[Theorem 5.F]{Halmos_m}). Take $A\in \RR$ and $(A_n)\subset \RR$ an increasing sequence with $\bigcup_{n} A_n=A$. Therefore, 
	$$
	(\s+\w)[\chi_A-\chi_{A_n}]=(\s+\w)[\chi_{A\backslash{A_n}}]=(\mu+\nu)(A\backslash{A_n})\to 0,
	$$
	because $\mu+\nu$ is a measure. This means that $\pi_{\s+\w}(\chi_{A_n})\to\pi_{\s+\w}(\chi_A)$ in $\H_{\s+\w}$ and, by continuity,
	\begin{align*}
		\mu_a(A)&=\n{P\pi_{\s+\w}(\chi_A)}_{\s+\w}^2=\lim_{n\to +\infty} \n{P\pi_{\s+\w}(\chi_{A_n})}_{\s+\w}^2=\lim_{n\to +\infty} \mu_a(A_n).
	\end{align*}
	In the same way, $\displaystyle\mu_s(A)=\lim_{n\to +\infty} \mu_s(A_n)$. Consequently, $\mu_a$ is a $\nu$-absolutely continuous measure and $\mu_s$ is a $\nu$-singular measure. 
	
	Before we move on the general case without any condition on the sign of $\mu$, we give an expression to the projector $P$. There exists $E\in \RR$ such that $\mu_s(E)=0$ and $\nu(E^c)=0$ 
	and, consequently, for all $A\in \RR$,
	\begin{align}	
	\label{def_sing_meas3}
	(\mu_a+\nu)(A\cap E)=(\mu_a+\nu)(A), \qquad \mu_s(A\cap E)=0.
	\end{align}
	Let $\phi,\psi\in \D$. Thus, $\phi\ol{\psi}=\sum_{k=1}^{n}a_k \chi_{A_k}$, for some $n\geq1$ and $A_k\in \RR$, disjoint subsets. Applying (\ref{def_sing_meas3})  we obtain that
	\begin{align*}
	(\s_a+\w)(\phi,\psi)&=\int_\mathcal{A} \phi \ol{\psi}d(\mu_a+\nu) =\sum_{k=1}^n a_i  (\mu_a+\nu)(A_k) \\
	&=\sum_{k=1}^n a_i  (\mu_a+\nu)(A_k \cap E) +\sum_{k=1}^n a_i  \mu_s(A_k \cap E)\\
	&=\int_\mathcal{A} \chi_E \phi \ol{\psi}d(\mu_a+\nu) +\int_\mathcal{A} \chi_E \phi \ol{\psi}d\mu_s \\
	&=\int_\mathcal{A} \chi_E \phi \ol{\psi}d(\mu+\nu).
	\end{align*}
	Clearly, $\chi_E \phi\in \D$. Therefore, we can write
	$$
	\pin{P\pi_{\s+\w}(\phi)}{\pi_{\s+\w}(\psi)}_{\s+\w}=(\s_a+\w)(\phi,\psi)=\pin{\pi_{\s+\w}(\chi_E\phi)}{\pi_{\s+\w}(\psi)}_{\s+\w}.
	$$
	Since $\D\backslash \ker(\s+\w)$ is dense in $\H_{\s+\w}$, we have 
	\begin{equation}
	\label{expr_P}
	P\pi_{\s+\w}(\phi)=\pi_{\s+\w}(\chi_E\phi), \qquad \forall \phi \in \D.
	\end{equation}
	
	Now, let $\mu$ be a (complex) measure on $(\mathcal{A}, \RR)$.
	Let $\t=\t_r+\t_m+\t_{ss}$ be the Lebesgue decomposition of $\t$ with respect to $\s\in M(\t)$ and $\w$.\\
	We can repeat the arguments above for $\s$ which is non-negative. Thus, $P$ act as in (\ref{expr_P}) with some $E\in \RR$. Taking into account (\ref{t_r}),
	\begin{align*}
	\t_r(\phi,\psi)&=\pin{TP\pi_{\s+\w}(\phi)}{P\pi_{\s+\w}(\psi)}_{\s+\w}\\
	&=\pin{T\pi_{\s+\w}(\chi_E\phi)}{\pi_{\s+\w}(\chi_E\psi)}_{\s+\w}\\
	&=\t(\chi_E\phi,\chi_E\psi)\\
	&=\int_\mathcal{A} \chi_E\phi\ol{\psi}d\mu, \qquad\qquad \forall \phi,\psi\in \D.
	\end{align*}
	Hence, $\mu_a(A):=\t_r[\chi_A]=\mu(A\cap E)$, for $A\in \RR$, is a measure on $(\mathcal{A},\RR)$. We can conclude that $\mu_a$ is $\nu$-absolutely continuous applying Lemma \ref{lem_reg_for_meas}. Also $\mu_s:=\mu-\mu_a$ is a measure on $(\mathcal{A},\RR)$ and, in particular, $\mu_s(A)=\t_s[\chi_A]=\mu(A\cap E^c)$. This shows that $\nu \perp \mu_s$.
\end{proof}

This proof does not involve the Jordan decomposition of a signed measure and, in the general case, it works also taking for $\s$ the sesquilinear form induced by any non-negative measure which bounds $\mu$.

\section*{Acknowledgments}

The author gratefully acknowledges the hospitality of the Alfréd Rényi Institute of Mathematics - Hungarian Academy of Sciences - and the Eötvös Loránd University of Budapest. He also wishes to thank Dr. T. Titkos and Dr. Zs. Tarcsay for many conversations. \\
This work was supported by the ''National Group for Mathematical Analysis, Probability and their Applications'' (GNAMPA – INdAM, project ''Problemi spettrali e di rappresentazione in quasi *-algebre di operatori'' 2017).

\vspace*{0.5cm}
\begin{center}
\textsc{Rosario Corso, Dipartimento di Matematica e Informatica} \\
\textsc{Università degli Studi di Palermo, I-90123 Palermo, Italy} \\
{\it E-mail address}: {\bf rosario.corso@studium.unict.it}
\end{center}

\end{document}